\documentclass[11pt]{amsart}
\usepackage{amssymb, latexsym, graphicx, mathrsfs}

\numberwithin{equation}{section}
\newtheorem{Thm}[equation]{Theorem}
\newtheorem{Prop}[equation]{Proposition}
\newtheorem{Cor}[equation]{Corollary}
\newtheorem{Lem}[equation]{Lemma}

\theoremstyle{definition}
\newtheorem{Def}[equation]{Definition}
\newtheorem{Exa}[equation]{Example}
\newtheorem{Rmk}[equation]{Remark}

\begin{document}

\title [Quantum affine algebras and $q$-deformation of
arithmetical functions]
{Quantum affine algebras, canonical bases and $q$-deformation of
arithmetical functions}
\author[Henry Kim]{Henry H. Kim$^{\star}$}
\thanks{$^{\star}$ partially supported by an NSERC grant.}
\address{Department of
Mathematics, University of Toronto, Toronto, ON M5S 2E4, CANADA}
\email{henrykim@math.toronto.edu}
\author[Kyu-Hwan Lee]{Kyu-Hwan Lee}
\address{Department of
Mathematics, University of Connecticut, Storrs, CT 06269, U.S.A.}
\email{khlee@math.uconn.edu}
\begin{abstract} 
In this paper, we obtain affine analogues of Gindikin-Karpelevich formula and Casselman-Shalika formula as sums over Kashiwara-Lusztig's canonical bases. Suggested by these formulas, we define natural $q$-deformation of arithmetical functions such as (multi-)partition function and Ramanujan $\tau$-function, and prove various identities among them. In some examples, we recover classical identities by taking limits. We also consider $q$-deformation of Kostant's function and study certain $q$-polynomials whose special values are  weight multiplicities.
\end{abstract}

\maketitle

\section*{Introduction} This paper is a continuation of \cite{KimLee-1}. The classical Gindikin-Karpelevich formula and Casselman-Shalika formula express certain integrals of spherical functions over maximal unipotent subgroups of $p$-adic groups as products over all positive roots. In \cite{KimLee-1}, we expressed the products over positive roots as sums over Kashiwara-Lusztig's canonical bases (\cite{Kash, Lusz-1}). That idea first appeared in the papers \cite{BN, McN} from the context of Weyl group multiple Dirichlet series \cite{BBFcomb, BBFcr}. (See also \cite{BBCFH, BBFH, BBF}.) 
Let $G$ be a split reductive $p$-adic group, $\chi$ be an unramified character of $T$, the maximal torus, and $f^0$ be the standard spherical vector corresponding to $\chi$. Let $\bold z$ be the element of 
${}^L T\subset {}^L G$, the $L$-group of $G$, corresponding to $\chi$ by the Satake isomorphism. Then
\begin{eqnarray} 
\int_{N_-(F)} f^0(n)\, dn &=& \prod_{\alpha\in \Delta^+} \frac {1-q^{-1}{\bold z}^{\alpha}}{1-{\bold z}^{\alpha}}=\sum_{b\in \mathbf B } (1-q^{-1})^{d(\phi_{\bold i}(b))}{\bold z}^{\mathrm{wt}(b)},  \\
\int_{N_-(F)} f^0(n)\psi_{\lambda}(n)\, dn & =&  \chi (V(\lambda))
 \prod_{\alpha\in \Delta^+} (1-q^{-1}{\bold z}^{\alpha}) \\ 
&=& (-t)^{M} \mathbf z^{ 2 \rho} \  \chi (V(\lambda))
 \prod_{\alpha\in \Delta^+} (1-t^{-1}{\bold z}^{-\alpha}) \nonumber
 \\ &=&   (-t)^{M} \mathbf z^{  \rho} \sum_{b' \otimes b \in \frak B_{\lambda} \otimes \frak B_{\rho} } G_{\rho}(b; t) \mathbf z^{\mathrm{wt}(b' \otimes b) },\nonumber
\end{eqnarray}
where $\Delta^+$ is the set of positive roots, $\bold B$ is the canonical basis, $\frak B_{\lambda}$ is the crystal basis with highest weight $\lambda$, and we set $M=| \Delta^+|$ and $t =q^{-1}$.  Notice that the tensor product of crystal bases behaves well in the Casselman-Shalika formula.

In the affine Kac-Moody groups, A. Braverman, D. Kazhdan and M. Patnaik calculated the integral (0.1) in \cite{BKP}, and obtained a formula of the form

\begin{equation} \label{eqn-correction}
\int_{N_-(F)} f^0(n)\, dn= A \prod_{\alpha\in \Delta^+} \left(\frac {1-q^{-1}{\bold z}^{\alpha}}{1-{\bold z}^{\alpha}}\right)^{\mathrm{mult}\, \alpha},
\end{equation}
where $A$ is a certain correction factor. When the underlying finite simple Lie algebra $\frak g_{\mathrm{cl}}$ is simply-laced of rank $n$, $A$ is given by
$$\prod_{i=1}^n \prod_{j=1}^{\infty} \frac {1-q^{-d_i}{\bold z}^{j\delta}}{1-q^{-d_i-1}{\bold z}^{j\delta}},
$$
where $d_i$'s are the exponents of $\frak g_{\mathrm{cl}}$, and $\delta$ is the minimal positive imaginary root.

In this paper, we use the explicit description of the canonical basis due to Beck, Chari, Pressley and Nakajima 
(\cite{BCP}, \cite{BeckNa}) to write the right hand side of (0.3) as a sum over the canonical basis. Moreover, we obtain the generalization of (0.2).
Namely, we prove (Theorem \ref{thm-first} and Corollary \ref{cor-familiar})
\begin{eqnarray}  \prod_{\alpha \in \Delta^+} \left ( \frac {1- q^{-1} \mathbf z^\alpha}{1 - \mathbf z^\alpha} \right )^{\mathrm{mult} \, \alpha} &=&  \sum_{b \in \mathbf B} (1 - q^{-1})^{d \left ( \phi (b) \right ) } \mathbf z^{\mathrm{wt}(b)}, \label{eqn-fff} \\
 \chi (V(\lambda)) \mathbf z^{\rho} \prod_{\alpha \in \Delta^+} (1 - q^{-1} \mathbf z^{- \alpha} )^{\mathrm{mult}\, \alpha}
&=&  \sum_{ b' \otimes b \in \frak B_{\lambda} \otimes \frak B_{\rho} } G_{\rho} (b; q) \, \mathbf z^{\mathrm{wt}(b' \otimes b)}, \label{eqn-ggg}
\end{eqnarray}
where $\mathbf B$ is the canonical basis of $\mathbf U^+$ (the positive part of the quantum affine algebra),
and $\frak B_{\lambda}$ is the crystal basis with highest weight $\lambda$. Here $\mathbf z$ is a formal variable. 
We also write the correction factor $A$ as a sum over a canonical basis in the case when $\frak g_{\mathrm{cl}}$ is simply-laced. We first prove (\ref{eqn-fff}) by induction, and deduce (\ref{eqn-ggg}) from (\ref{eqn-fff}) and Weyl-Kac character formula. In the course of proof, we see that (\ref{eqn-ggg}) can be considered as a $q$-deformation of Weyl-Kac character formula. 

In the development of Weyl group multiple Dirichlet series (\cite{BBCFH, BBFH, BBF, BBFcomb, BBFcr}), one of the main problems has been how to define the local coefficient (or $p$-part). Various combinatorial methods have been adopted to define the coefficient. In particular, the string parametrization (or BZL-path) of a crystal graph was used in \cite{BBFcr}. In our previous paper \cite{KimLee-1}, the $q$-polynomial  $H_{\lambda+\rho}(\mu;q)$ was defined using the generating series:
\begin{equation} \label{eqn-abc} \sum_{\mu \in Q_+} H_{\lambda}(\mu; q ) \mathbf z^{\lambda - \mu}  = \sum_{w \in {W}} (-1)^{\ell(w)} \sum_{b \in \mathbf B}  (1-q^{-1})^{d(\phi (b) ) } \mathbf z^{w \lambda- \mathrm{wt}(b) } .\end{equation} This definition works for all the finite root systems in a uniform way, and is essentially the same as the local coefficient in the non-metaplectic case that was introduced in the work \cite{CGco} of Chinta and Gunnells on the construction of Weyl group multiple Dirichlet series. 

In this paper, we define the polynomial $H_{\lambda+\rho}(\mu;q) \in \Bbb Z[q^{-1}]$ for affine cases using the same formula. That is to say, the definition (\ref{eqn-abc}) works for affine cases as well. Moreover, we show that the polynomial $H_{\lambda+\rho}(\mu;q)$ has many remarkable, representation-theoretic properties; its constant term is the multiplicity of the weight $\lambda-\mu$ in $V(\lambda)$, and the value at $q=-1$ is the multiplicity of the weight $\lambda+\rho-\mu$ in the tensor product $V(\lambda)\otimes V(\rho)$. See Corollary \ref{cor-HHH}. It is also related to Kazhdan-Lusztig polynomials when $\frak g$ is of finite type (Corollary \ref{cor-KL}).

Our construction also has connections to deformations of Kostant's partition function. When $q=-1$ and $\lambda$ is a strictly dominant weight, the Casselman-Shalika formula (0.5) gives a formula for multiplicity of the weight $\nu$ in the tensor product $V(\lambda-\rho)\otimes V(\rho)$
in terms of $q$-deformation of Kostant partition function, generalizing the result of \cite[Theorem 1]{GR} to affine Kac-Moody algebras (See (\ref{eqn-GR})). More precisely, we define
$K^\infty_q(\mu)$, in a similar way as in \cite{GR}, by
\begin{equation*}
\sum_{\mu\in Q_+} K^\infty_q(\mu) {\mathbf z}^\mu=\prod_{\alpha \in \Delta^+} \left ( \frac {1- q^{-1} \mathbf z^\alpha}{1 - \mathbf z^\alpha} \right )^{\mathrm{mult} \, \alpha}.
\end{equation*} 
Note that when $q=\infty$, $K^\infty_{q}(\mu)$ is the classical Kostant partition function. Then we have

\begin{equation*} 
\dim \,  (V(\lambda-\rho)\otimes V(\rho))_{\nu}=\sum_{w\in W} (-1)^{l(w)} K^\infty_{-1}(w\lambda -\nu).
\end{equation*}

Another application is deformation of arithmetical functions. Since the set of positive roots is infinite, the left-hand sides of (0.4) and (0.5) become infinite products. It leads to natural $q$-deformation of arithmetical functions such as multi-partition functions and Fourier coefficients of modular forms. We indicate one example here.

We define $\epsilon_{q,n}(k)$ as
\[ \prod_{k=1}^\infty (1 -q^{-1} t^k)^n = \sum_{k=0}^\infty \epsilon_{q,n}(k) t^k .\] 
Note that $\epsilon_{1,n}(k)$ is a classical arithmetic function related to modular forms. For example, we have $\epsilon_{1, 24}(k) = \tau(k+1)$, where $\tau(k)$ is the Ramanujan $\tau$-function. Thus the function $\epsilon_{q,n}(k)$ should be considered as a $q$-deformation of the function $\epsilon_{1,n} (k)$.

For a multi-partition $\mathbf p = (\rho^{(1)}, \dots , \rho^{(n)} ) \in \mathcal P(n)$, we define \[ p_{q, n} (k) = \sum_{\mathbf p \in \mathcal P(n) \atop |\mathbf p|=k} (1-q^{-1})^{d(\mathbf p )} ,   \qquad k \ge 1 ,\] and set $p_{q, n}(0)=1$. Here $|\mathbf p|$ is the weight of the multi-partition and the number $d(\mathbf p)$ is defined in Section 1. Notice that if $q \rightarrow \infty$ and $k>0$, the function $p_{\infty, n}(k)$ is nothing but the multi-partition function with $n$-components. In particular, $p_{\infty,1}(k)=p(k)$, the usual partition function.
Hence we can think of $p_{q, n} (k)$ as a $q$-deformation of the multi-partition function.

It turned out that there are interesting relations among these $q$-deformations.  We prove (Proposition \ref{prop-recur})
$$\epsilon_{q,n}(k) = \sum_{r=0}^k \epsilon_{1,n} (r) p_{q, n}(k-r),
$$ 
which yields an infinite family of $q$-polynomial identities. 
We also obtain ``classical" identities by taking limits.
When  $n=24$ and $q\to\infty$, the identity becomes a well-known recurrence formula for the Ramanujan $\tau$-function:
\[ 0= \sum_{r=0}^k \tau(r+1) p_{\infty, 24} (k-r) . \]

In fact, we  prove another family of identities (Proposition \ref{prop-cl-H}) and obtain an intriguing characterization of the function $\epsilon_{q, n}(k)$.  In Example \ref{exa-cl-H}, by taking $q=1$, we write $\tau(k+1)$ as a sum of certain numbers arising from the structure of the affine Lie algebra of type $A^{(1)}_4$. To be precise, we have
$$\tau(k+1)= \lim_{q \rightarrow 1} \sum_{\mu \in Q_{+,\mathrm{cl}}} H_\rho(k \alpha_0 + \mu;q ) \Big /  (1-q^{-1})^{10},
$$ where $Q_{+,\mathrm{cl}}$ is the classical nonnegative root lattice of type $A_4$.

\subsection*{Acknowledgments} We would like to thank M. Patnaik for explaining his results [4]. We also thank A. Ram, P. Gunnells, S. Friedberg and B. Brubaker for their useful comments on this paper. Finally, we thank Seok-Jin Kang and G. Benkart for their interest in this work.

\vskip 1cm

\section{Gindikin-Karpelevich Formula}

Let $\frak g$ be an untwisted affine Kac-Moody algebra over $\mathbb C$. We denote by $I = \{ 0, 1, \dots , n \}$ the set of indices for simple roots. Let $W$ be the Weyl group. We keep almost all the notations in Sections 2 and 3 of \cite{BeckNa}. However, we use $v$ for the parameter of a quantum group and reserve $q$ for another parameter. 
Whenever there is a discrepancy in notations, we will make it clear. 

We fix $\mathbf h =( \dots , i_{-1}, i_0, i_1, \dots)$ as in Section 3.1 in \cite{BeckNa}. Then for any integers $m < k$, the product $s_{i_m} s_{i_{m+1}} \cdots s_{i_k} \in W$ is a reduced expression, so is the product  $s_{i_k} s_{i_{k-1}} \cdots s_{i_m} \in W$. We set \[ \beta_k = \begin{cases} s_{i_0}s_{i_{-1}} \cdots s_{i_{k+1}} ( \alpha_{i_k}) & \text{ if } k \le 0 , \\ s_{i_1}s_{i_{2}} \cdots s_{i_{k-1}} ( \alpha_{i_k}) & \text{ if } k > 0 , \end{cases} \] and define
\[ \mathscr{R}(k) = \{ \beta_0 , \beta_{-1}, \dots , \beta_{k} \} \ \text{ for } k \le 0 \quad \text{ and } \quad  \mathscr{R}(k) = \{ \beta_1 , \beta_{2}, \dots , \beta_{k} \} \ \text{ for } k > 0 .\]

Let $T_i=T^{''}_{i, 1}$ be the automorphism of $\mathbf U$ as in Section 37.1.3. of \cite{Lusz-Bk}, and let 
\[ \mathbf c_+= (c_0, c_{-1}, c_{-2} , \dots ) \in \mathbb N^{\mathbb Z_{\le 0} } \quad \text{ and } \quad \mathbf c_- = (c_1, c_2, \dots ) \in \mathbb N^{\mathbb Z_{> 0} }\] be functions (or sequences) that are almost everywhere zero. We denote by $\mathscr C_> $ (resp. by $\mathscr C_<$) the set of such functions $\mathbf c_+$ (resp. $\mathbf c_-$).
Then we define
\[ E_{\mathbf c_+} = E_{i_0}^{(c_0 )} T^{-1}_{i_0} \left ( E_{i_{-1}}^{( c_{-1} )} \right )  T^{-1}_{i_0}  T^{-1}_{i_{-1}}  \left ( E_{i_{-2} }^{(  c_{-2}  )} \right ) \cdots \] and \[ E_{\mathbf c_-} = \cdots T_{i_1}  T_{i_{2}}  \left ( E_{i_{3} }^{(  c_{3}  )} \right ) T_{i_1} \left ( E_{i_{2}}^{( c_{2} )} \right )   E_{i_1}^{(c_1 )}  . \] 
We set 
$$ B(k)=\begin{cases} \{ E_{\mathbf c_+} : c_m =0 \text{ for } m <k  \}  & \text{ for } k \le 0, \\
                       \{ E_{\mathbf c_-} : c_m =0 \text{ for } m > k  \} & \text{ for } k > 0 .
\end{cases}
$$

We denote by $\mathbf B$ the Kashiwara-Lusztig's canonical bases for $\mathbf U^+$, the positive part of the quantum affine algebra.

\begin{Prop} \cite{BCP, BeckNa} For each $E_{\mathbf c_+} \in B(k),\ k \le 0$ (resp. $E_{\mathbf c_-} \in B(k), \ k > 0$), there exists a unique $b \in \mathbf B$ such that  \begin{equation} \label{eqn-ec} b \equiv E_{\mathbf c_{+} } \text{(resp. $E_{\mathbf c_{-} } $)} \mod \ v^{-1} \mathbb Z[v^{-1}].\end{equation}
\end{Prop}

We denote by $\mathbf B(k)$ the subset of $\mathbf B$ corresponding to $B(k)$ as in the above theorem. Then we define the map $\phi : \mathbf B(k) \rightarrow \mathscr C_>$ for $k \le 0$ (resp. $\mathscr C_<$ for $k>0$) to be  $b \mapsto \mathbf c_+ $ (resp. $\mathbf c_- $) such that the condition (\ref{eqn-ec}) holds. For an element $\mathbf c_+=(c_0, c_{-1},  \dots ) \in \mathscr C_>$ (resp. $\mathbf c_-=(c_1, c_2, \dots ) \in \mathscr C_>$), we define $d(\mathbf c_+)$ (resp. $d(\mathbf c_-)$) to be the number of nonzero $c_i$'s. 

\begin{Prop} \label{prop-Weyl}
For each $k \in \mathbb Z$, we have
\begin{equation} \label{eqn-id} \prod_{\alpha \in \mathscr R(k)} \frac {1- q^{-1} \mathbf z^\alpha}{1 - \mathbf z^\alpha} =  \sum_{b \in \mathbf B(k)} (1 - q^{-1})^{d \left ( \phi (b) \right ) } \mathbf z^{\mathrm{wt}(b)}.
\end{equation}
\end{Prop}

\begin{proof} First we assume $k>0$ and use induction on $k$. If $k=1$, then the identity (\ref{eqn-id}) is easily verified. Now, using an induction argument, we obtain
\begin{eqnarray*}
\prod_{\alpha \in \mathscr R (k) } \frac {1- q^{-1} \mathbf z^\alpha}{1 - \mathbf z^\alpha} &=& \left ( \prod_{\alpha \in \mathscr R (k-1) } \frac {1- q^{-1} \mathbf z^\alpha}{1 - \mathbf z^\alpha} \right ) \frac {1- q^{-1} \mathbf z^{\beta_k} }{1 - \mathbf z^{\beta_k } } \\ &=&
 \left (  \sum_{b \in \mathbf B(k-1)} (1 - q^{-1})^{d \left ( \phi(b) \right ) } \mathbf z^{ \mathrm{wt}(b)} \right ) \left ( 1+ \sum_{j \ge 1 } (1-q^{-1}) \mathbf z^{j \beta_k} \right ) \\ &=& \sum_{b \in \mathbf B(k-1)} (1 - q^{-1})^{d \left ( \phi (b) \right ) } \mathbf z^{\mathrm{wt}(b)} + \sum_{j \ge 1 } \sum_{b \in \mathbf B(k-1)} (1 - q^{-1})^{d \left ( \phi (b) \right )+1 } \mathbf z^{\mathrm{wt}(b) +  j \beta_k } .
\end{eqnarray*}
 On the other hand, since $b' \in \mathbf B(k)$ satisfies \[ b' \equiv b \    T_{i_1} T_{i_2} \cdots T_{i_{k} }\left ( E_{k}^{(j)} \right ) \ \mathrm{mod} \ v^{-1} \mathbb Z[v^{-1}] \] for unique $b \in \mathbf B(k-1)$ and $j \ge 0$, we can write $\mathbf B(k)$ as a disjoint union \[ \mathbf B(k) = \bigcup_{j \ge 0} \{ b' \in \mathbf B(k) \ | \ \phi (b') = ( c_1, \dots, c_{k-1} , j,0 ,0, \dots ), \ c_i \in \mathbb N \}. \] Now it is clear that
 \begin{eqnarray*}
& & \sum_{b \in \mathbf B(k) } (1 - q^{-1})^{d \left ( \phi (b) \right ) } \mathbf z^{ \mathrm{wt}(b)} \\ &=& \sum_{b \in \mathbf B(k-1)} (1 - q^{-1})^{d \left ( \phi (b) \right ) } \mathbf z^{\mathrm{wt}(b)} + \sum_{j \ge 1 } \sum_{b \in \mathbf B(k-1)} (1 - q^{-1})^{d \left ( \phi (b) \right )+1 } \mathbf z^{\mathrm{wt}(b) +  j \beta_k} .
\end{eqnarray*} This completes the proof of the case $k>0$.

The case $k \le 0$ can be proved in a similar way through a downward induction.
\end{proof}

We set \[ \mathscr{R}_{>} = \bigcup_{k \le 0} \mathscr{R}(k) \quad \text{ and } \quad   \mathscr{R}_{<} = \bigcup_{k >0} \mathscr{R}(k).\]
Similarly, we put
\[ \mathbf B_{>} = \bigcup_{k \le 0} \mathbf B(k) \quad \text{ and } \quad   \mathbf B_{<} = \bigcup_{k >0} \mathbf B(k).\]

\begin{Cor} \label{cor-1}
We have
\begin{equation} \label{eqn-id-cor} \prod_{\alpha \in \mathscr R_>} \frac {1- q^{-1} \mathbf z^\alpha}{1 - \mathbf z^\alpha} =  \sum_{b \in \mathbf B_>} (1 - q^{-1})^{d \left ( \phi (b) \right ) } \mathbf z^{\mathrm{wt}(b)}.
\end{equation}
The same identity is true if $\mathscr R_>$ and $\mathbf B_>$ are replaced with  $\mathscr R_<$ and $\mathbf B_<$, respectively.
\end{Cor}

Let $\mathbf c_0 = (\rho^{(1)}, \rho^{(2)}, \dots , \rho^{(n)} )$ be a multi-partition with $n$ components, {\it i.e.} each component $\rho^{(i)}$ is a partition. We denote by $\mathcal P(n)$ the set of all multi-partitions with $n$ components.
Let $S_{\mathbf c_0}$ be defined as in \cite{BeckNa} (p. 352) and set \[ B_0 = \{ S_{\mathbf c_0} \ |\  \mathbf c_0 \in \mathcal P(n) \} .\]

\begin{Prop} \cite{BCP, BeckNa} For each $S_{\mathbf c_0} \in B_0$, there exists a unique $b \in \mathbf B$ such that  \begin{equation} \label{eqn-ec0} b \equiv S_{\mathbf c_{0} } \mod \ v^{-1} \mathbb Z[v^{-1}].\end{equation}
\end{Prop}

We denote by $\mathbf B_0$ the subset of $\mathbf B$ corresponding to $B_0$.  Using the same notation $\phi$ as we used for $\mathbf B(k)$, we define a function $\phi : \mathbf B_0 \rightarrow \mathcal P(n)$, $b \mapsto \mathbf c_0$,  such that the condition (\ref{eqn-ec0}) is satisfied. 
 
 For a partition $\mathbf p = (1^{m_1}2^{m_2} \cdots r^{m_r} \cdots)$, we define \[ d(\mathbf p) = \mathrm\#\{ r \, | \, m_r \neq 0 \} \quad \text { and } \quad |\mathbf p |= m_1+2 m_2+ 3 m_3 + \cdots .\] Then for a multi-partition $\mathbf c_0 = (\rho^{(1)}, \rho^{(2)}, \dots , \rho^{(n)} ) \in \mathcal P(n)$, we set \[ d(\mathbf c_0)= d(\rho^{(1)} ) + d(\rho^{(2)} ) + \cdot + d(\rho^{(n)} ). \] We obtain from the definition of $S_{\mathbf c_0}$ that if $\phi(b)= \mathbf c_0$ then \[ \mathrm{wt}(b) = | \mathbf c_0| \delta ,\] where $| \mathbf c_0 | = |\rho^{(1)}| + \cdots + |\rho^{(n)}|$ is the weight of the multi-partition $\mathbf c_0$.

\begin{Prop} \label{prop-imag}
We have 
\begin{equation} \label{eqn-main-1} \prod_{\alpha \in \Delta^+_{\mathrm{im} } } \left ( \frac {1- q^{-1} \mathbf z^\alpha}{1 - \mathbf z^\alpha} \right )^{\mathrm{mult} \, \alpha} =  \prod_{k=1}^\infty \left ( \frac {1- q^{-1} \mathbf z^{k \delta} }{1 - \mathbf z^{k \delta}} \right )^{n} =  \sum_{b \in \mathbf B_0} (1 - q^{-1})^{d \left ( \phi (b) \right ) } \mathbf z^{\mathrm{wt}(b)}, \end{equation}
where $\Delta^+_{\mathrm{im}}$ is the set of positive imaginary roots of $\mathfrak g$.
\end{Prop}

\begin{proof}
The first equality follows from the facts $\Delta^+_{\mathrm {im}} = \{ \delta , 2 \delta, 3 \delta , \dots \}$ and $\mathrm{mult}(k \delta ) =n$ for all $k =1, 2, \dots$. Now we consider the second equality and assume $n=1$. Then we have
\begin{equation} \label {eqn-delta}
\prod_{k=1}^\infty \left ( \frac {1- q^{-1} \mathbf z^{k \delta} }{1 - \mathbf z^{k \delta}} \right ) = \prod_{k=1}^\infty \left ( 1 + \sum_{j=1}^\infty (1 - q^{-1}) \mathbf z^{jk\delta} \right ).
\end{equation}
We consider the generating function of the partition function $p(m)$:

\begin{equation} \sum_{m=0}^\infty p(m) \mathbf z^{m \delta}=\prod_{k=1}^\infty \left ( 1 + \sum_{j=1}^\infty  \mathbf z^{jk\delta} \right ) = \sum_{\rho^{(1)} \in \mathcal P(1)} \mathbf z^{| \rho^{(1)} | \delta} = \sum_{b \in \mathbf B_0} \mathbf z^{\mathrm{wt}(b)}. \label{eqn-delta-1}
\end{equation}

Comparing (\ref{eqn-delta}) and (\ref{eqn-delta-1}), we see that if we expand the product in the right-hand side of (\ref{eqn-delta}) into a sum, the coefficient of $\mathbf z^{| \rho^{(1)} | \delta}$   will be a power of $(1-q^{-1})$ and that the exponent of $(1-q^{-1})$ is exactly the number $d (\rho^{(1)})$. Therefore, we obtain 
\begin{eqnarray*}
\prod_{k=1}^\infty \left ( \frac {1- q^{-1} \mathbf z^{k \delta} }{1 - \mathbf z^{k \delta}} \right )  &=& 
\sum_{\rho^{(1)} \in \mathcal P(1)} (1-q^{-1})^{d(\rho^{(1)}) } \mathbf z^{|\rho^{(1)}| \delta} \\ &=& \sum_{b \in \mathbf B_0}  (1-q^{-1})^{d(b) } \mathbf z^{\mathrm{wt}(b)} .
\end{eqnarray*}

Next we assume that $n=2$. Then we have
\begin{eqnarray*}
& & \prod_{k=1}^\infty \left ( \frac {1- q^{-1} \mathbf z^{k \delta} }{1 - \mathbf z^{k \delta}} \right )^2 \\&=& \left ( \sum_{\rho^{(1)} \in \mathcal P(1)} (1-q^{-1})^{d(\rho^{(1)}) } \mathbf z^{|\rho^{(1)}| \delta} \right ) \left ( \sum_{\rho^{(2)} \in \mathcal P(1)} (1-q^{-1})^{d(\rho^{(2)}) } \mathbf z^{|\rho^{(2)}| \delta} \right ) \\ &=& \sum_{(\rho^{(1)}, \rho^{(2)}) \in \mathcal P(2)} (1 -q^{-1})^{d(\rho^{(1)}) + d(\rho^{(2)}) } \mathbf z^{\left ( | \rho^{(1)} | +  | \rho^{(2)} | \right ) \delta } \\ &=& \sum_{b \in \mathbf B_0}  (1-q^{-1})^{d(b) } \mathbf z^{\mathrm{wt}(b)} . 
\end{eqnarray*}
It is now clear that this argument naturally generalizes to the case $n>2$.
\end{proof}

\medskip

Let us consider the correction factor $A$ in (\ref{eqn-correction}). We will make a modification of the formula (\ref{eqn-main-1}) to write $A$ as a sum over $\mathbf B_0$ in the case when the underlying classical Lie algebra $\frak g_{\mathrm{cl}}$ is simply-laced. For a partition $\mathbf p =(1^{m_1}2^{m_2} \cdots )$ and $d_i \in \mathbb N$, we define 
\[ Q_{d_i}(\mathbf p, j ) = \begin{cases} (1 - q) q^{-(d_i+1) m_j } & \text{if } m_j \neq 0,  \\ 1 &  \text{if } m_j = 0 ,  \end{cases} \quad \text{ and } \quad Q_{d_i}(\mathbf p) = \prod_{j=1}^\infty Q_{d_i}(\mathbf p, j ).\]
For a multi-partition $\mathbf p=(\rho^{(1)}, \dots , \rho^{(n)})$ and $d_i \in \mathbb N$, $i= 1, \dots , n$, we define 
\[ Q_{d_1, \dots, d_n}(\mathbf p)= \prod_{i=1}^n  \quad Q_{d_i}(\rho^{(i)}) .\]  Then we obtain the following.

\begin{Cor} Assume that $\frak g_{\mathrm{cl}}$ is simply-laced. Then we have
$$A= \prod_{i=1}^n \prod_{j=1}^{\infty} \frac {1-q^{-d_i}{\bold z}^{j\delta}}{1-q^{-d_i-1}{\bold z}^{j\delta}} =  \sum_{b \in \mathbf B_0}Q(\phi(b)) \mathbf z^{\mathrm{wt}(b)},
$$
where $d_i$'s are the exponents of $\frak g_{\mathrm{cl}}$ and we write $Q (\mathbf p) = Q_{d_1, \dots, d_n}(\mathbf p)$.
\end{Cor}

\begin{proof}
The first equality is a result in \cite{BKP}. The second equality can be obtained using a similar argument as in the proof of Proposition \ref{prop-imag}.
\end{proof}

\medskip

Let $\mathscr C= \mathscr C_> \times \mathcal P(n) \times \mathscr C_<$ as in \cite{BeckNa}. 
\begin{Thm} \cite{BCP, BeckNa}  \label{thm-BeNa} There is a bijection between the sets $\mathbf B$ and $\mathscr C$ such that for each $\mathbf c = (\mathbf c_+, \mathbf c_0, \mathbf c_-) \in \mathscr C$, there exists a unique $b \in \mathbf B$ such that  \begin{equation} \label{eqn-ec-final} b \equiv E_{\mathbf c_+} S_{\mathbf c_{0} } E_{\mathbf c_-} \mod \ v^{-1} \mathbb Z[v^{-1}].\end{equation}
\end{Thm}
Then we naturally extend the function $\phi$ to a bijection of $\mathbf B$ onto $\mathscr C$ and the number $d(\mathbf c)$ is also defined by $d(\mathbf c)=d(\mathbf c_+)+d(\mathbf c_0)+d(\mathbf c_-)$ for each $\mathbf c \in \mathscr C$.

\begin{Thm} \label{thm-first} We have 
\begin{equation} \label{eqn-main-2} \prod_{\alpha \in \Delta^+} \left ( \frac {1- q^{-1} \mathbf z^\alpha}{1 - \mathbf z^\alpha} \right )^{\mathrm{mult} \, \alpha} =  \sum_{b \in \mathbf B} (1 - q^{-1})^{d \left ( \phi (b) \right ) } \mathbf z^{\mathrm{wt}(b)}.
\end{equation}
\end{Thm}

\begin{proof}
Recall that $\Delta^+ = \Delta^+_{\mathrm{re}} \cup \Delta^+_{\mathrm{im}}$, $\Delta^+_{\mathrm{re}} = \mathscr R_> \cup \mathscr R_<$ and $\mathrm{mult} \ \alpha =1$ for $\alpha \in \Delta^+_{\mathrm{re}}$. Then the identity of the theorem follows from Corollary \ref{cor-1}, Proposition \ref{prop-imag} and Theorem \ref{thm-BeNa}.
\end{proof}

\medskip

\section{Casselman-Shalika Formula}

For the functions 
$\mathbf c_+= (c_0, c_{-1}, c_{-2} , \dots ) \in \mathscr C_>$ and $\mathbf c_- = (c_1, c_2, \dots ) \in \mathscr C_<$, we define \[ | \mathbf c_+ | = c_0 + c_{-1}+c_{-2}+\cdots \quad \text{ and } \quad | \mathbf c_- |= c_1 +c_2 + \cdots .\] For a multi-partition $\mathbf c_0 = (\rho^{(1)}, \rho^{(2)}, \dots , \rho^{(n)} ) \in \mathcal P(n)$, we set $| \mathbf c_0 | = |\rho^{(1)}| + \cdots + |\rho^{(n)}|$, as we did in the previous section.

Using similar arguments in the previous section, we obtain the following identities. 
\begin{Prop} \label{prop-collection}  \hfill
\begin{enumerate}
 \item We have for each $k \in \mathbb Z$, \begin{equation*}  \prod_{\alpha \in \mathscr R(k)} ( 1- q^{-1} \mathbf z^\alpha )^{-1} =  \sum_{b \in \mathbf B(k)}  q^{-| \phi (b)  | } \mathbf z^{\mathrm{wt}(b)}.
\end{equation*}

\item 
\begin{equation*} \prod_{\alpha \in \mathscr R_>} ( {1- q^{-1} \mathbf z^\alpha})^{-1} =  \sum_{b \in \mathbf B_>}  q^{-| \phi (b) | } \mathbf z^{\mathrm{wt}(b)}.
\end{equation*}
The same identity is true if $\mathscr R_>$ and $\mathbf B_>$ are replaced with  $\mathscr R_<$ and $\mathbf B_<$, respectively.

\item
 \begin{equation*}  \prod_{\alpha \in \Delta^+_{\mathrm{im} } } \left ( 1- q^{-1} \mathbf z^\alpha \right )^{- \mathrm{mult} \, \alpha} =  \prod_{k=1}^\infty \left ( 1- q^{-1} \mathbf z^{k \delta}  \right )^{-n} =  \sum_{b \in \mathbf B_0} q^{-| \phi (b) | } \mathbf z^{\mathrm{wt}(b)}. \end{equation*}

\item \begin{equation*}  \prod_{\alpha \in \Delta^+} \left ( 1- q^{-1} \mathbf z^\alpha \right )^{-\mathrm{mult} \, \alpha} =  \sum_{b \in \mathbf B} q^{-| \phi (b) | } \mathbf z^{\mathrm{wt}(b)}.
\end{equation*}
\end{enumerate}
\end{Prop}

\medskip

Let $P_+= \{ \lambda \in P | \langle  h_i , \lambda \rangle \ge 0 \text{ for all } i \in I \}$. Recall that  the irreducible $\frak g$-module $V(\lambda)$ is integrable if and only if  $\lambda \in P_+$ (\cite{Kac}, Lemma 10.1).

\begin{Def} \label{def-a} Let $\lambda \in P_+$.  We define $H_{\lambda}(\cdot; q) : Q_+ \rightarrow \mathbb Z[q^{-1}]$ using the generating series \begin{eqnarray*} \sum_{\mu \in Q_+} H_{\lambda}(\mu; q ) \mathbf z^{\lambda - \mu}  &=& \sum_{w \in {W}} (-1)^{\ell(w)} \sum_{b \in \mathbf B}  (1-q^{-1})^{d(\phi (b) ) } \mathbf z^{w \lambda- \mathrm{wt}(b) } \\ &=& \left ( \sum_{w \in {W}} (-1)^{\ell(w)} {\mathbf z}^{w\lambda} \right ) \left ( \sum_{b \in \mathbf B}  (1-q^{-1})^{d(\phi (b) ) } \mathbf z^{- \mathrm{wt}(b) } \right ), \end{eqnarray*} and we write
 \[ \chi_{q} (V(\lambda)) = \sum_{\mu \in Q_+} H_{\lambda }(\mu; q) \ \mathbf z^{\lambda - \mu}. \] \end{Def}

We denote by $\chi(V(\lambda) )$ the usual character of $V(\lambda)$. We have the element $d \in \frak h$ such that $\alpha_0 (d) = 1$ and $\alpha_j(d)=0$, $j \in I \setminus \{ 0 \}$. We define $\rho \in \frak h^*$ as in \cite[chapter 6]{Kac} by $\rho (h_j)=1$, $j \in I$ and $\rho (d)=0$. By the Weyl-Kac character formula,
\begin{equation*}
\frac {\displaystyle\sum_{w \in {W} } (-1)^{\ell (w) } \mathbf z^{w (\lambda+ \rho)  - \rho} }{ \displaystyle\prod_{\alpha \in \Delta^+} ( 1 - \mathbf z^{-\alpha} )^{\mathrm{mult}\, \alpha} } = \chi (V(\lambda) ).
\end{equation*}

In particular, if $\lambda=0$, then
$$ 
\sum_{w \in {W}} (-1)^{\ell (w)} {\mathbf z}^{w\rho}
={\mathbf z}^{\rho}\prod_{\alpha \in \Delta^+} ( 1 - {\mathbf z}^{-\alpha})^{\mathrm{mult}\, \alpha}.
$$
By Theorem \ref{thm-first}, 
$$\sum_{b \in \mathbf B}  (1-q^{-1})^{d(\phi (b) ) } \mathbf z^{- \mathrm{wt}(b) }=\prod_{\alpha \in \Delta^+} \left ( \frac {1- q^{-1} \mathbf z^{-\alpha} }{1 - \mathbf z^{-\alpha} } \right )^{\mathrm{mult} \, \alpha}.
$$

Thus we obtain
\begin{eqnarray*}\chi_q(V(\rho)) &=&\left ( \sum_{w \in {W}} (-1)^{\ell(w)} {\mathbf z}^{w\rho} \right ) \left ( \sum_{b \in \mathbf B}  (1-q^{-1})^{d(\phi (b) ) } \mathbf z^{- \mathrm{wt}(b) } \right ) \\ &=& {\mathbf z}^{\rho}\prod_{\alpha \in \Delta^+} ( 1 - {\mathbf z}^{-\alpha})^{\mathrm{mult}\, \alpha} \prod_{\alpha \in \Delta^+} \left ( \frac {1- q^{-1} \mathbf z^{-\alpha} }{1 - \mathbf z^{-\alpha} } \right )^{\mathrm{mult} \, \alpha} \\ &=& \mathbf z^{\rho} \prod_{\alpha \in \Delta^+} (1 - q^{-1} \mathbf z^{- \alpha} )^{\mathrm{mult}\, \alpha}  .
\end{eqnarray*}

Therefore we have proved the following. 

\begin{equation} \label{eqn-important} \chi_{q} (V(\rho) ) = \mathbf z^{\rho} \prod_{\alpha \in \Delta^+} (1 - q^{-1} \mathbf z^{- \alpha} )^{\mathrm{mult}\, \alpha} . \end{equation} 

When $q=-1$ in (\ref{eqn-important}), we have  the following identity by \cite[Exercise 10.1]{Kac}.
\begin{Lem} \label{lem-ord} \[ \chi_{-1}(V(\rho))=\mathbf z^{\rho} \prod_{\alpha \in \Delta^+} (1+ \mathbf z^{- \alpha} )^{\mathrm{mult}\, \alpha} =\chi(V(\rho)).\]
\end{Lem}

\begin{Rmk} \label{rho} 
By Definition \ref{def-a}, 
$$\chi_{-1}(V(\rho))=\sum_{\mu\in Q_+} H_{\rho}(\mu; -1) {\mathbf z}^{\rho-\mu}=\mathbf z^{\rho} \prod_{\alpha \in \Delta^+} (1+ \mathbf z^{- \alpha} )^{\mathrm{mult}\, \alpha}.
$$
Therefore, if $H_\rho(\mu; -1)\ne 0$, $\rho -\mu$ must be a weight of $V(\rho)$ and $H_\rho(\mu; -1)$ is the multiplicity of $\rho-\mu$ in $V(\rho)$. 
\end{Rmk}

Now we have the following proposition which is an affine analogue of the Casselman-Shalika formula.

\begin{Prop} \label{prop-aa}
\begin{equation} \label{eqn-prod} \chi_q (V({\lambda + \rho})) = \chi (V(\lambda)) \chi_q (V(\rho)) . \end{equation}
\end{Prop}
\begin{proof} By Definition \ref{def-a} and Theorem \ref{thm-first},
$$\chi_q(V(\lambda+\rho))=
\left ( \sum_{w \in {W}} (-1)^{\ell(w)} {\mathbf z}^{w(\lambda+\rho)} \right ) \prod_{\alpha \in \Delta^+} \left ( \frac {1- q^{-1} \mathbf z^{-\alpha} }{1 - \mathbf z^{-\alpha} } \right )^{\mathrm{mult} \, \alpha}.
$$
By the Weyl-Kac character formula and (\ref{eqn-important}), the right hand side is $\chi(V(\lambda))\chi_q(V(\rho))$.
\end{proof}

\begin{Rmk} When $q=1$, we see that $\chi_1(V(\lambda+\rho))\mathbf z^{-\rho}$ is the numerator of the Weyl-Kac character formula. Hence we can think of (\ref{eqn-prod}) as a $q$-deformation of Weyl-Kac character formula. 
\end{Rmk}

At special values of $q$, the formal sum $\chi_q(V(\lambda+\rho))$ becomes characters of some representations as you can see in the following corollary.

\begin{Cor} \label{cor-tensor} \hfill
\begin{enumerate}
\item When $q=\infty$, we have $$\chi_{\infty}(V(\lambda+\rho))={\mathbf z}^{\rho}\chi(V(\lambda)).
$$
Hence we may consider $\chi_q(V(\lambda+\rho))\mathbf z^{-\rho}$ as a $q$-deformation of $\chi (V(\lambda))$. 

\item When $q=-1$, we obtain
\begin{eqnarray*} \chi_{-1}(V(\lambda+\rho)) = \chi(V(\lambda )) \chi (V(\rho))=\chi(V(\lambda)\otimes V(\rho)).
\end{eqnarray*}
\end{enumerate}
\end{Cor}

\begin{proof}
(1) The identity is true since $\chi_{\infty}(V(\rho))={\mathbf z}^{\rho}$.
(2) By putting $q=-1$ in (\ref{eqn-prod}), the identity follows from Lemma \ref{lem-ord}.
\end{proof}

We also obtain representation-theoretic meaning of special values of $H_{\lambda +\rho}(\mu;q)$.

\begin{Cor} \label{cor-HHH} \hfill
\begin{enumerate} 
\item The value $H_{\lambda+\rho}(\mu; \infty)$ is the multiplicity of the weight $\lambda-\mu$ in $V(\lambda)$.
\item The value $H_{\lambda+\rho}(\mu; -1)$ is the multiplicity of the weight $\lambda+\rho-\mu$ in the tensor product $V(\lambda)\otimes V(\rho)$.
\end{enumerate}
\end{Cor}

\begin{proof}
(1) By Definition \ref{def-a} and from Corollary \ref{cor-tensor} (1), we have
$$
\sum_{\mu\in Q_+} H_{\lambda+\rho}(\mu; \infty) {\mathbf z}^{\lambda-\mu}= \mathbf z^{-\rho} \chi_{\infty}(V(\lambda+\rho))= \chi(V(\lambda)).
$$
This proves the part (1).

(2) We obtain from Corollary \ref{cor-tensor} (2)
\[  \chi_{-1}(V(\lambda+\rho)) = \sum_{\mu\in Q_+} H_{\lambda+\rho}(\mu; -1) {\mathbf z}^{\lambda+\rho-\mu} =\chi(V(\lambda)\otimes V(\rho)) ,\] which proves the part (2).
\end{proof}


Before we further investigate the implication of the Casselman-Shalika formula (\ref{eqn-prod}), we need the following lemma.

\begin{Lem} \label{lem-weights} 
Assume that $\lambda_1, \lambda_2 \in P_+$. Then the set of weights of $V(\lambda_1) \otimes V(\lambda_2)$ is the same as that of 
$V(\lambda_1 + \lambda_2)$.
\end{Lem}

\begin{proof}
Suppose that $\lambda_1,\lambda_2\in P_+$. Let $V(\lambda_1)$ and $V(\lambda_2)$ be the integrable highest weight modules with highest weights $\lambda_1$ and $\lambda_2$, respectively.
By \cite[p. 211]{Kac}, $V(\lambda_1+\lambda_2)$ occurs in $V(\lambda_1)\otimes V(\lambda_2)$ with multiplicity one. Hence it is enough to prove that
any weight of $V(\lambda_1)\otimes V(\lambda_2)$ is a weight of $V(\lambda_1+\lambda_2)$.

If $V_1$ and $V_2$ are modules in the category $\mathcal{O}$, then the weight space of $(V_1\otimes V_2)_{\mu}$ for $\mu\in \frak h^*$, 
is given by
$$(V_1\otimes V_2)_{\mu}=\sum_{\nu\in \frak h^*} (V_1)_{\nu}\otimes (V_2)_{\mu-\nu}.
$$
Hence weights of $V(\lambda_1)\otimes V(\lambda_2)$ are of the form $\mu_1+\mu_2$, where $\mu_1$ and $\mu_2$ are weights of $V(\lambda_1)$ and $V(\lambda_2)$, respectively. Furthermore, since  $V(\lambda_1)\otimes V(\lambda_2)$ is completely reducible, a weight $\mu_1+\mu_2$ of $V(\lambda_1)\otimes V(\lambda_2)$ is a weight of the module $V(\lambda)$ for some $\lambda \in P_+$, that appears in the decomposition of $V(\lambda_1)\otimes V(\lambda_2)$.

It follows from Corollary 10.1 in \cite{Kac} that we can choose $w\in W$ such that $w(\mu_1+\mu_2)\in P_+$. Then, by Proposition 11.2 in \cite{Kac}, we need only to show that $w(\mu_1+\mu_2)$ is nondegenerate with respect to $\lambda_1+\lambda_2$. 
By Lemma 11.2 in \cite{Kac}, $w\mu_1$ and $w\mu_2$ are nondegenerate with respect to $\lambda_1$ and $\lambda_2$, respectively. Now, from the definition of nondegeneracy \cite[p.190]{Kac}, we see that
$w\mu_1+w\mu_2$ is nondegenerate with respect to $\lambda_1+\lambda_2$. 
\end{proof}

Now we use crystal bases, namely, bases at $v=0$, since they behave nicely under tensor products.
Let $\frak B_\lambda$ be the crystal basis associated to a dominant integral weight $\lambda \in P_+$.
We  choose $G_{\rho}(\cdot; q): \frak B_\rho \rightarrow  \mathbb Z [q^{-1}]$  by assigning any element of $\mathbb Z [q^{-1}]$ to each $b \in \frak B_\rho$ so that 
\begin{equation} \label{eqn-hh} H_\rho (\mu; q) = \sum_{b \in \frak B_\rho \atop \mathrm{wt}(b) = \rho - \mu} G_\rho(b; q) .\end{equation} 
By Remark \ref{rho}, it is enough to consider $\mu \in Q_+$
such that $\rho -\mu$ is a weight of $b \in \frak B_{\rho}$. 

Using the function $G_{\rho}(\cdot; q)$, we can rewrite Casselman-Shalika formula in Proposition \ref{prop-aa} in a familiar form:
\begin{Cor} \label{cor-familiar}
\begin{equation} \label{eqn-tensor} 
\sum_{\mu \in Q_+} H_{\lambda + \rho}(\mu; q) \, \mathbf z^{\lambda + \rho - \mu} =
\chi (V(\lambda)) \mathbf z^{\rho} \prod_{\alpha \in \Delta^+} (1 - q^{-1} \mathbf z^{- \alpha} )^{\mathrm{mult}\, \alpha}
=  \sum_{ b' \otimes b \in \frak B_{\lambda} \otimes \frak B_{\rho} } G_{\rho} (b; q) \, \mathbf z^{\mathrm{wt}(b' \otimes b)}.
\end{equation} 
\end{Cor}

\begin{proof}
The first equality is obvious from (\ref{eqn-important}) and Proposition \ref{prop-aa}. For the second equality, we obtain
\begin{eqnarray*} 
& & \chi (V(\lambda)) \mathbf z^{\rho} \prod_{\alpha \in \Delta^+} (1 - q^{-1} \mathbf z^{- \alpha} )^{\mathrm{mult} \, \alpha} = \chi (V(\lambda)) \chi_q( V(\rho)) \\ &=&  \left ( \sum_{b' \in \frak B_\lambda} \mathbf z^{\mathrm{wt}(b')} \right ) \left (   \sum_{\mu \in Q_+} H_{ \rho}( \mu; q) \, \mathbf z^{ \rho - \mu} \right ) =  \left ( \sum_{b' \in \frak B_\lambda} \mathbf z^{\mathrm{wt}(b')} \right ) \left (   \sum_{b \in \frak B_\rho} G_{ \rho}(b; q)  \mathbf z^{\mathrm{wt}(b)}  \right )\\ & =& \sum_{ b' \otimes b \in \frak B_{\lambda} \otimes \frak B_{\rho} } G_{\rho} (b; q) \, \mathbf z^{\mathrm{wt}(b' \otimes b)}.
\end{eqnarray*}
\end{proof}


The following proposition  provides useful information on $H_{\lambda +\rho} (\mu; q)\in \Bbb Z[q^{-1}]$. 

\begin{Prop} Assume that $\lambda \in P_+$. Then we have
$H_{\lambda +\rho} (\mu; q)$ is a nonzero polynomial if and only if $\lambda + \rho - \mu $ is a weight of $V(\lambda +\rho)$.
\end{Prop}

\begin{proof}
We obtain from (\ref{eqn-tensor}) that if $H_{\lambda +\rho} (\mu; q)  \neq 0$ then $\lambda + \rho - \mu $ is a weight of $V(\lambda) \otimes V(\rho)$. Then $\lambda + \rho - \mu $ is a weight of $V(\lambda +\rho)$ by Lemma \ref{lem-weights}. Conversely, 
assume that $\lambda + \rho - \mu $ is a weight of $V(\lambda +\rho)$, so a weight of  $V(\lambda) \otimes V(\rho)$.
By Corollary \ref{cor-tensor} (2), 
\[ \sum_{\mu' \in Q_+} H_{\lambda + \rho} (\mu'; -1) \mathbf z^{\lambda + \rho - \mu'}  = \chi (V(\lambda) \otimes V(\rho)) .\] 
Since $\lambda + \rho - \mu $ is a weight of $V(\lambda) \otimes V(\rho)$, the coefficient $H_{\lambda + \rho} (\mu; -1) \neq 0$. 
Then  $H_{\lambda + \rho} (\mu; q)$ is a nonzero polynomial.
\end{proof}

\medskip

\section{Applications}

We give several applications of our formulas to $q$-deformation of (multi-)partition functions and modular forms, and Kostant's function and multiplicity formula. We also obtain formulas for $H_\lambda(\mu;q)$.

\subsection{multi-partition functions and modular forms}

We will write $\mathcal P = \mathcal P(1)$. For a partition $\mathbf p = (1^{m_1} 2^{m_2} \cdots r^{m_r} \cdots ) \in \mathcal P$, we define  \[ \kappa_q (\mathbf p) = \begin{cases} (-q^{-1})^{\sum m_r} & \text{ if } m_r=0 \text{ or }1 \text{ for all } r ,\\ 0 & \text{ otherwise.} \end{cases} \] We define for $k \ge 1$ \[ \epsilon_{q} (k) = \sum_{\mathbf p \in \mathcal P \atop |\mathbf p|=k} \kappa_q (\mathbf p) ,\]  and set $\epsilon_q(0)=1$.
For example, we have $\epsilon_q(5)= 2 q^{-2} - q^{-1}$ and $\epsilon_q(6)=-q^{-3} + 2 q^{-2} -q^{-1}$.

From the definitions, we have
\[ \prod_{k=1}^\infty (1 -q^{-1} t^k) = 1+ \sum_{\mathbf p \in \mathcal P} \kappa_q (\mathbf p)  t^{|\mathbf p |} = 1+ \sum_{k=1}^\infty \epsilon_q(k) t^k .\] 
Then it follows from Euler's Pentagonal Number Theorem that when $q=1$, we have
\begin{equation} 
\label{eqn-euler} \epsilon_1(k ) =  \begin{cases} (-1)^m & \text{ if } k=\frac 1 2 m(3m \pm 1) ,\\ 0 & \text{ otherwise.} \end{cases}
\end{equation}

We also define for $k \ge 1$ \[ p_q (k) = \sum_{\mathbf p \in \mathcal P \atop |\mathbf p|=k} (1-q^{-1})^{d(\mathbf p )} ,\] where $d(\mathbf p )$ is the same as in the previous sections, and we set $p_q(0)=1$. Note that if $k>0$, $p_{\infty}(k)=p(k)$. Hence we can think of $p_q(k)$ as a $q$-deformation of the partition function.

\begin{Prop} \label{prop-partition}
If $k>0$, then
\begin{equation}\label{eqn-prop-3} \epsilon_q(k)- p_q(k) = \sum_{m=1}^\infty (-1)^m \left \{ p_q ( k - \frac 1 2 m(3m-1)) +  p_q ( k - \frac 1 2 m(3m+1))  \right \} , \end{equation}
where we define $p_q(M)=0$ for all negative integer $M$.
\end{Prop}

\begin{proof}
We put $n=1$ in Proposition \ref{prop-imag} and obtain
\[  \prod_{k=1}^\infty (1 - q^{-1} \mathbf z^{k \delta} ) = \left ( \sum_{\mathbf p \in \mathcal P} (1 - q^{-1})^{d(\mathbf p)} \mathbf z^{|\mathbf p| \delta} \right ) \prod_{k=1}^\infty (1 - \mathbf z^{k \delta} ) .\] After the change of variables $\mathbf z^{\delta}  = t$, we obtain
\begin{eqnarray*} 
1+ \sum_{k=1}^\infty \epsilon_q(k) t^k   &=& \prod_{k=1}^\infty (1 - q^{-1} t^k  ) \\ &=& \left ( \sum_{\mathbf p \in \mathcal P} (1 - q^{-1})^{d(\mathbf p )} t^{|\mathbf p|} \right ) \prod_{k=1}^\infty (1 - t^{k} ) \\ &=&  \left ( 1+ \sum_{k=1}^\infty p_q (k) t^k  \right ) \left ( 1+ \sum_{m=1}^\infty    (-1)^m  \left \{  t^{\frac 1 2 m(3m-1)} + t^{\frac 1 2 m(3m+1) } \right \} \right ) , 
\end{eqnarray*}
where we use the definition of $p_q(k)$ and (\ref{eqn-euler}) in the last equality.
We obtain the identity (\ref{eqn-prop-3}) by expanding the product and equating the coefficient of $t^k$ with $\epsilon_q(k)$.
\end{proof}

As a corollary of the proof of Proposition \ref{prop-partition}, we obtain the following.
\begin{Cor} Let $(a;q)_n=\prod_{k=0}^{n-1} (1-a q^k)$. Then
$$\sum_{n=0}^{\infty} \frac {(q^{-1};t)_n}{(t;t)_n} \ t^n=\sum_{k=0}^\infty p_q(k) t^k.
$$
\end{Cor}

\begin{proof} By the $q$-binomial theorem, 
$$\prod_{k=1}^\infty (1-q^{-1} t^k)= \left (\sum_{n=0}^\infty \frac {(q^{-1};t)_n}{(t;t)_n} \ t^n \right ) \prod_{k=1}^\infty (1-t^k).
$$
Comparing this with the identity in the proof of Proposition 3.2, we obtain the result. 
\end{proof}

\begin{Rmk} When $q\to\infty$, we have
$$\sum_{n=0}^{\infty} \frac {t^n}{(t;t)_n} =\sum_{\mathbf p\in \mathcal P} t^{|\mathbf p|}=\sum_{n=0}^\infty p(n) t^n.
$$
This is a special case of \cite[Corollary 2.2]{An}.
\end{Rmk}

\medskip

We generalize Proposition \ref{prop-partition} to the case of multi-partitions.  For a multi-partition $\mathbf p = (\rho^{(1)}, \dots , \rho^{(n)} ) \in \mathcal P(n)$, we define  \[ \kappa_q (\mathbf p) =  \prod_{i=1}^n \kappa_q (\rho^{(i)}) , \] and for $k \ge 1$ \begin{equation} \label{eqn-eps} \epsilon_{q,n} (k) = \sum_{\mathbf p \in \mathcal P(n) \atop |\mathbf p|=k} \kappa_q (\mathbf p) ,\end{equation}  and set $ \epsilon_{q,n}(0)=1$.
From the definitions, we have
\[ \prod_{k=1}^\infty (1 -q^{-1} t^k)^n = 1+ \sum_{\mathbf p \in \mathcal P(n)} \kappa_q (\mathbf p)  t^{|\mathbf p |} = \sum_{k=0}^\infty \epsilon_{q,n}(k) t^k .\] 
One can see that if $k>0$, we have $\epsilon_{\infty,n}(k)=0$. 

\begin{Rmk} Note that $\epsilon_{1,n}(k)$ is a classical arithmetic function related to modular forms. For example, we have $\epsilon_{1, 24}(k) = \tau(k+1)$, where $\tau(k)$ is the Ramanujan $\tau$-function. Thus the function $\epsilon_{q,n}(k)$ should be considered as a $q$-deformation of the function $\epsilon_{1,n} (k)$.
\end{Rmk}

We also define for $k \ge 1$ \[ p_{q, n} (k) = \sum_{\mathbf p \in \mathcal P(n) \atop |\mathbf p|=k} (1-q^{-1})^{d(\mathbf p )} ,\] and set $p_{q, n}(0)=1$. Notice that if $k>0$, the function $p_{\infty, n}(k)$ is nothing but the multi-partition function with $n$-components. Hence we can think of $p_{q, n} (k)$ as a $q$-deformation of the multi-partition function.

\begin{Prop} \label{prop-recur}
If $k>0$, then
\begin{equation} \label{eqn-prop-4} \epsilon_{q,n}(k) = \sum_{r=0}^k \epsilon_{1,n} (r) p_{q, n}(k-r) . \end{equation}
\end{Prop}

\begin{proof}
We obtain from Proposition \ref{prop-imag}
\[  \prod_{k=1}^\infty (1 - q^{-1} \mathbf z^{k \delta} )^n = \left ( \sum_{\mathbf p \in \mathcal P(n)} (1 - q^{-1})^{d(\mathbf p)} \mathbf z^{|\mathbf p| \delta} \right ) \prod_{k=1}^\infty (1 - \mathbf z^{k \delta} )^n .\] After the change of variables $\mathbf z^{\delta}  = t$, we obtain from the definitions
\begin{eqnarray*} 
\sum_{k=0}^\infty \epsilon_{q,n}(k) t^k   &=& \left ( \sum_{\mathbf p \in \mathcal P(n)} (1 - q^{-1})^{d(\mathbf p )} t^{|\mathbf p|} \right ) \prod_{k=1}^\infty (1 - t^{k} )^n \\ &=&  \left (  \sum_{r=0}^\infty p_{q,n} (r) t^r  \right ) \left (  \sum_{s=0}^\infty    \epsilon_{1,n}(s) t^s  \right ) .
\end{eqnarray*} Now the identity (\ref{eqn-prop-4}) is clear.
\end{proof}

By taking $q\to\infty$, we obtain the following identity
$$0=\sum_{r=0}^k \epsilon_{1,n}(r) p_{\infty, n}(k-r),\qquad k>0,
$$ 
where $p_{\infty, n}(k)$ is the multi-partition function with $n$-components. This classical identity is also a consequence of the following identities:
$$\prod_{k=1}^\infty (1 -t^k)^n = \sum_{k=0}^\infty \epsilon_{1,n}(k) t^k \quad \text{ and } \quad \prod_{k=1}^\infty (1 -t^k)^{-n} = \sum_{k=0}^\infty p_{\infty,n}(k) t^k.
$$

\begin{Exa}
When the affine Kac-Moody algebra $\frak g$ is of type $X^{(1)}_{24}$, $X=A, B, C$ or $D$, we  have
\[  \epsilon_{q,24}(k) = \sum_{r=0}^k \tau (r+1) p_{q, 24}(k-r) \] and
 \[ 0= \sum_{r=0}^k \tau(r+1) p_{\infty, 24} (k-r) ,\] where $\tau(k)$ is the Ramanujan $\tau$-function.
If $k=2$, the first identity becomes
 \[ \epsilon_{q, 24} (2) = \tau(1) p_{q, 24}(2) +\tau(2) p_{q, 24}(1) +\tau(3) p_{q, 24}(0) .  \] Through some computations, we obtain \[  \epsilon_{q, 24} (2) = 276 q^{-2}   -24 q^{-1}.   \] On the other hand, we have
\begin{eqnarray*}  & & \tau(1) p_{q, 24}(2) +\tau(2) p_{q, 24}(1) +\tau(3) p_{q, 24}(0) \\ &=&
p_{q, 24}(2)  -24 p_{q, 24}(1) +252 \\ &=&  \{ 276 (1-q^{-1})^2  + 48 (1-q^{-1})  \} -24 \cdot 24 (1-q^{-1}) +252 \\ &=&  276(1-q^{-1})^2 -528 (1-q^{-1}) + 252 \\ & =&  276 q^{-2}   -24 q^{-1} = \epsilon_{q, 24} (2) .\end{eqnarray*} We also see that 
\begin{equation*}  \tau(1) p_{\infty , 24}(2) +\tau(2) p_{\infty, 24}(1) +\tau(3) p_{\infty , 24}(0) = 324 - 24 \cdot 24 + 252 =0 .  \end{equation*}
\end{Exa}

\smallskip

Now we consider the whole set of positive roots, not just the set of imaginary positive roots, and obtain interesting identities. We begin with the identity (\ref{eqn-important}). Recalling the description of the set of positive roots, we obtain
\begin{eqnarray} \sum_{\mu \in Q_+} H_\rho (\mu;q) \mathbf z^{-\mu} &=& \mathbf z^{-\rho} \chi_q (V(\rho)) = \prod_{\alpha \in \Delta_+} (1 - q^{-1} \mathbf z^{-\alpha} )^{\mathrm{mult} \, \alpha} \nonumber \\ &=& \left ( \prod_{k=1}^\infty (1 - q^{-1} \mathbf z^{-k \delta} )^{n} \prod_{\alpha \in \Delta_{\mathrm{cl} } }(1 - q^{-1} \mathbf z^{\alpha -k \delta} ) \right ) \prod_{\alpha \in \Delta^+_{\mathrm{cl} } } (1 - q^{-1} \mathbf z^{-\alpha} ), \label{eqn-mmd}
\end{eqnarray}
where $\Delta_{\mathrm{cl}}$ is the set of classical roots.

Let $\mathcal Z = \{ \sum_{\alpha \in Q_+} c_\alpha \mathbf z^{-\alpha} \, | \, c_\alpha \in \mathbb C \}$ be the set of (infinite) formal sums. Recall that we have the element $d \in \frak h$ such that $\alpha_0 (d) = 1$ and $\alpha_j(d)=0$, $j \in I \setminus \{ 0 \}$. Let $\frak h_{\mathbb Z}$ be the $\mathbb Z$-span of $\{ h_0, h_1, \dots , h_n, d \}$.
We define the evaluation map $EV_t: \mathcal Z \times \frak h_{\mathbb Z} \rightarrow \mathbb C[[t]]$ by
\[ EV_t \left ( \sum_\alpha c_\alpha \mathbf z^{-\alpha} , \mathbf s \right ) = \sum_\alpha c_\alpha t^{\alpha ( \mathbf s)}, \qquad \mathbf s \in \frak h_{\mathbb Z} .\]   Then we see that $EV_t(\cdot , d)$ is the same as the {\em basic specialization} in \cite[p.219]{Kac} with $q$ replaced by $t$. We apply $EV_t(\cdot, d)$ to (\ref{eqn-mmd}) and obtain
\begin{equation} \label{eqn-rrr} (1-q^{-1})^{|\Delta^+_{\mathrm{cl} }|}
\prod_{k=1}^\infty (1 -q^{-1}t^k)^{\dim \frak g_{\mathrm{cl}}} = \sum_{k=0}^\infty \left ( \sum_{\mu \in Q_{+,\mathrm{cl}}} H_\rho(k \alpha_0 + \mu ;q) \right ) t^k , \end{equation} where $\frak g_{\mathrm{cl}}$ is the finite-dimensional simple Lie algebra corresponding to $\frak g$, and $Q_{+,\mathrm{cl}}$ is the $\mathbb Z_{\ge 0}$-span of $\{ \alpha_1, \dots , \alpha_n \}$. We write $|\Delta^+_{\mathrm{cl} }| =r$ and  $\dim \frak g_{\mathrm{cl}}=N$ so that $N=2r+n$.
By comparing (\ref{eqn-rrr}) with the identity $\prod_{k=1}^\infty (1-q^{-1} t^k)^n=\sum_{k=0}^\infty \epsilon_{q,n}(k) t^k$, we obtain:

\begin{Prop} \label{prop-cl-H}
\[ \epsilon_{q,N}(k) = \sum_{\mu \in Q_{+,\mathrm{cl}}} H_\rho(k \alpha_0 + \mu;q ) \Big /  (1-q^{-1})^r . \]
\end{Prop}

By Definition \ref{def-a}, $\epsilon_{q,N}(k)$ is a power series in $q^{-1}$ in the above formula. However, one can see from (\ref{eqn-eps}) that $\epsilon_{q,N}(k)$ is actually a polynomial in $q^{-1}$.

\begin{Exa} \label{exa-cl-H}
We take $\frak g$ to be of type $A^{(1)}_4$. Then the classical Lie algebra $\frak g_{\mathrm{cl}}$ is of type $A_4$, and $r=|\Delta^+_{\mathrm{cl} }|=10$ and $N=\dim \frak g_{\mathrm{cl}}=24$. Taking the limit $q \rightarrow 1$, we obtain
$$\tau(k+1)= \lim_{q \rightarrow 1} \sum_{\mu \in Q_{+,\mathrm{cl}}} H_\rho(k \alpha_0 + \mu;q ) \Big /  (1-q^{-1})^{10}.
$$ Therefore the sum  $\sum_{\mu \in Q_{+,\mathrm{cl}}} H_\rho(k \alpha_0 + \mu;q )$ is always divisible by $(1-q^{-1})^{10}$. However, the famous Lehmer's conjecture predicts that the sum is never divisible by $(1-q^{-1})^{11}$.
\end{Exa}

\subsection{Kostant's function and the polynomial $H_\lambda(\mu;q)$}

In this subsection, let $\frak g$ be a finite-dimensional simple Lie algebra (finite type) or an untwisted affine Kac-Moody algebra (affine type). 

\begin{Def}
We define the functions $K^\infty_q(\mu)$ and $K^1_q(\mu)$ by 
\[ \sum_{\mu \in Q_+} K^\infty_{q}(\mu) \mathbf z^\mu = \prod_{\alpha \in \Delta_+} \left ( \frac {1 - q^{-1} \mathbf z^\alpha } { 1 - \mathbf z^{\alpha} } \right )^{\mathrm{mult}\, \alpha} = \sum_{b \in \mathbf G} (1-q^{-1})^{d(\phi(b)}\mathbf z^{\mathrm{wt}(b)}\] and \[ \sum_{\mu \in Q_+} K^1_{q}(\mu) \mathbf z^\mu = \prod_{\alpha \in \Delta_+} \left ( {1 - q^{-1} \mathbf z^\alpha } \right )^{-\mathrm{mult}\, \alpha} = \sum_{b \in \mathbf G} q^{-|\phi(b)|} \mathbf z^{\mathrm{wt}(b)}. \]  We set $K^\infty_q(\mu)=K^1_q(\mu)=0$ if $\mu\not\in Q_+$.
\end{Def}

\begin{Rmk} \hfill
\begin{enumerate}
\item Note that both $K^\infty_\infty(\mu)$ with $q=\infty$ and $K^1_1 (\mu)$ with $q=1$ are equal to the classical Kostant's partition function $K(\mu)$. Hence both of them can be considered as $q$-deformation of Kostant's function.
\item The function $K^1_q(\mu)$ was introduced by Lusztig \cite{Lusz-sing} for finite types. See also S. Kato's paper \cite{Kato}. On the other hand, the function $K^\infty_q(\mu)$ for finite types can be found in the work of Guillemin and Rassart \cite{GR}.
\end{enumerate}
\end{Rmk}

We obtain from the Casselman-Shalika formula (Proposition \ref{prop-aa})
\begin{eqnarray*}  \mathbf z^{- \lambda} \chi(V(\lambda)) &=& \sum_{\beta \in Q_+} (\dim V(\lambda)_{\lambda -\beta} )\mathbf z^{-\beta} \\ &=&  \mathbf z^{- \lambda -\rho } \chi_q(V(\lambda+\rho)) \prod_{\alpha \in \Delta_+} (1 -  q^{-1} z^{-\alpha} )^{- \mathrm{mult}\, \alpha}  \\ &=& \left ( \sum_{\mu \in Q_+} H_{\lambda +\rho}(\mu ;q) \mathbf z^{-\mu} \right ) \left (  \sum_{\nu \in Q_+} K^1_{q}(\nu) \mathbf z^{-\nu}  \right) .
\end{eqnarray*}
Therefore, we have a $q$-deformation of the Kostant's multiplicity formula:
\begin{Prop} \label{prop-KMF}
\[ \dim V(\lambda)_{\lambda-\beta} = \sum_{\mu \in Q_+} H_{\lambda +\rho}(\mu;q)K^1_{q}(\beta -\mu) .\]
\end{Prop}

In order to see that this is indeed a $q$-deformation of the Kostant's multiplicity formula, we need to
determine the value of $H_{\lambda+\rho}(\mu; 1)$. 

\begin{Lem} \label{lem-H-1}  We have \[H_{\lambda +\rho}(\mu;1) = \begin{cases}(-1)^{\ell(w)} &\text{ if } w \circ \lambda = - \mu \text{ for some } w \in W , \\ \quad 0 & \text{ otherwise,}
\end{cases} \] where we define $w \circ \lambda = w (\lambda + \rho) -\lambda -\rho$ for $w \in W$ and $\lambda \in P_+$.
\end{Lem} 
Note that such an element $w \in W$ is unique if it exists, so there is no ambiguity in the assertion.
\begin{proof}
From Definition \ref{def-a}, we obtain
\[ \sum_{\mu \in Q_+} H_{\lambda +\rho}(\mu;1) \mathbf z^{\lambda + \rho - \mu} = \sum_{w \in W} (-1)^{\ell(w)} \mathbf z^{w(\lambda +\rho)}. \] The condition $\lambda + \rho - \mu = w(\lambda +\rho)$ is equivalent to $w \circ \lambda = -\mu$. It completes the proof.
\end{proof}

Now we take $q=1$ in Proposition \ref{prop-KMF} and use Lemma \ref{lem-H-1} to obtain the classical Kostant's multiplicity formula
\[ \dim V(\lambda)_{\lambda-\beta} =\sum_{w \in W} (-1)^{\ell(w)} K(w \circ \lambda + \beta ). \]
Note that the sum is actually a finite sum. Indeed, we have $w \circ \lambda <0$ for each $w \in W$ and  $w \circ \lambda +\beta \ge 0$ only for finitely many $w \in W$ for fixed $\lambda \in P_+$ and $\beta \in Q_+$. For the same reason, the sum in (\ref{eqn-K-infty}) below is also a finite sum.

\begin{Rmk} We have obtained in the previous section (Corollary \ref{cor-HHH})
\begin{eqnarray}
H_{\lambda +\rho}(\mu ;\infty) &=& \dim V(\lambda)_{\lambda - \mu}, \label{eqn-H-inf} \\
H_{\lambda +\rho}(\mu ;-1) &=& \dim (V(\lambda) \otimes V(\rho))_{\lambda + \rho -\mu} . \label{eqn-H-minus}
\end{eqnarray}
When $\frak g$ is of finite type, we defined $H_{\lambda}(\mu; q)$ in \cite{KimLee-1} as in Definition 2.2, and we can prove the analogous results.
\end{Rmk}

We derive a formula for $H_{\lambda+\rho} (\mu;q)$ in the following proposition.
\begin{Prop}  \label{prop-K-infty}
 \begin{equation}  \label{eqn-K-infty}
 H_{\lambda+\rho} (\mu;q) =  \sum_{w \in W} (-1)^{\ell(w)} K^\infty_{q}(w \circ \lambda + \mu).
\end{equation}
\end{Prop}

\begin{proof}
We have from the definitions
\begin{eqnarray*} 
\chi_q(V(\lambda+\rho)) &=& \sum_{\mu \in Q_+} H_{\lambda +\rho}(\mu; q) \mathbf z^{\lambda+\rho-\mu} \\ &=& \left ( \sum_{w \in W} (-1)^{\ell(w)} \mathbf z^{w(\lambda+\rho)} \right ) \left (  \sum_{\nu \in Q_+} K^\infty_{q}(\nu) \mathbf z^{-\nu} \right ).
\end{eqnarray*} The identity in the proposition follows from expanding the product and comparing the coefficients.
\end{proof}

If we take the limit $q \rightarrow \infty$ in (\ref{eqn-K-infty}), we have from (\ref{eqn-H-inf})
\[ \dim  V(\lambda)_{\lambda - \mu} =  \sum_{w \in W} (-1)^{\ell(w)} K(w \circ \lambda + \mu),\]
which is again the classical Kostant's multiplicity formula.

If we take $q=-1$ in (\ref{eqn-K-infty}), we obtain from (\ref{eqn-H-minus}) \begin{equation} \label{eqn-GR}  \dim (V(\lambda) \otimes V(\rho))_{\lambda + \rho -\mu} =  \sum_{w \in W} (-1)^{\ell(w)} K^\infty_{-1}(w \circ \lambda + \mu).\end{equation} This is  a generalization of the formula in Theorem 1 of \cite{GR}  to the affine case.

\begin{Exa}
Assume that  $\frak g$ is of type $A_1^{(1)}$. We write $\mu =  m \alpha_0 + n \alpha_1 = (m,n) \in Q_+$ and set $\lambda =0$ in (\ref{eqn-K-infty}). Through standard computation, we obtain that  \[ \{ w \rho  + \mu -\rho \, | \, w \in W \} = \left \{ (m- \frac {k(k+1)} 2 , n - \frac {k(k-1)} 2) \, \Big | \, k \in \mathbb Z \right \} . \] 
Thus we have \[ H_\rho (m,n;q) = \sum_{k \in \mathbb Z} (-1)^k K^\infty_q (m- \frac {k(k+1)} 2 , n - \frac {k(k-1)} 2) .\]
By taking the limit $q \rightarrow \infty$, we obtain for $(m,n) \neq (0,0)$ \[ 0= \sum_{k \in \mathbb Z} (-1)^k K (m- \frac {k(k+1)} 2 , n - \frac {k(k-1)} 2) .\] In this case, $K(m,n)$ counts the number of vector partitions of $(m,n)$ into parts of the forms $(a,a)$, $(a-1, a)$ or $(a, a-1)$. Then one can see that we have obtained (3.9) on p.148 in \cite{Carlitz}.
\end{Exa}

\medskip

We further investigate properties of the function $H_{\lambda}(\mu ; q)$. 
From the definitions of $K^\infty_q(\mu)$ and $K^1_q(\mu)$, we have 
\begin{eqnarray*}  \left ( \sum_{\mu \in Q_+} K^\infty_{q}(\mu) \mathbf z^\mu \right ) \left (\sum_{\nu \in Q_+} K^1_{q}(\nu) \mathbf z^\nu \right )   &=& \prod_{\alpha \in \Delta_+} \left ( \frac {1 - q^{-1} \mathbf z^\alpha } { 1 - \mathbf z^{\alpha} } \right )^{\mathrm{mult}\, \alpha} \prod_{\alpha \in \Delta_+} \left ( {1 - q^{-1} \mathbf z^\alpha } \right )^{-\mathrm{mult}\, \alpha} \\ & =&  \prod_{\alpha \in \Delta_+} \left ( {1 - \mathbf z^\alpha } \right )^{-\mathrm{mult}\, \alpha} =  \sum_{\beta \in Q_+} K(\beta) \mathbf z^\beta , \end{eqnarray*} where $K(\beta)$ is the classical Kostant's function. Thus we have
\begin{equation} \sum_{\mu \in Q_+} K^\infty_q(\mu) K^1_q(\beta - \mu) = K(\beta), \end{equation} and we obtain, for $\beta > 0$,
\begin{equation} K^\infty_q(\beta) = K(\beta) -K^1_q(\beta) - \sum_{0 < \nu <\beta} K^\infty_q(\nu) K^1_q(\beta - \nu) , \end{equation} and $K^\infty_q(0)=K^1_q(0)=K(0)=1$.

 Then we obtain from Proposition \ref{prop-K-infty}
\begin{eqnarray*} H_{\lambda+\rho} (\mu;q) &=&   H_{\lambda+\rho} (\mu;1)+  \sum_{w \in W} (-1)^{\ell(w)} K(w \circ \lambda + \mu )  - \sum_{w \in W} (-1)^{\ell(w)} K^1_q(w \circ \lambda + \mu )  \\ & & \phantom{LLLLLLLLLLLLLL} - \sum_{w \in W \atop w \circ \lambda + \mu >0} (-1)^{\ell(w)} \sum_{0 < \nu <w \circ \lambda + \mu } K^\infty_q(\nu ) K^1_q(w \circ \lambda + \mu  -\nu), \end{eqnarray*} where $H_{\lambda+\rho} (\mu;1)$ plays the role of correction term for the case $w \circ \lambda + \mu =0$. See Lemma \ref{lem-H-1} for the value of $H_{\lambda+\rho} (\mu;1)$. Also we used the fact that $K(\beta)=K_q^1(\beta)=K_q^\infty(\beta)=0$ unless $\beta \ge 0$.

Now we apply the classical Kostant formula and get the following proposition.
\begin{Prop}\label{prop-WW} Assume that $\lambda \in P_+$ and $\mu \in Q_+$. Then we have
\begin{eqnarray*} H_{\lambda+\rho} (\mu;q) &=&   H_{\lambda+\rho} (\mu;1)+\dim V(\lambda)_{\lambda - \mu} - \sum_{w \in W} (-1)^{\ell(w)} K^1_q(w \circ \lambda +\mu)  \\ & & \phantom{LLLLLLLLLLLLLL} - \sum_{w \in W \atop w \circ \lambda + \mu >0} (-1)^{\ell(w)}  \sum_{0 < \nu <w \circ \lambda +\mu} K^\infty_q(\nu ) K^1_q(w \circ \lambda +\mu -\nu).\end{eqnarray*}
\end{Prop}

\medskip

For the rest of this section, we assume that $\frak g$ is of finite type. We denote by $\rho^\vee$ the element of $\frak h$ defined by $\langle \alpha_i, \rho^\vee \rangle=1$ for all the simple roots $\alpha_i$. The following identity was conjectured by Lusztig  \cite{Lusz-sing} and proved by S. Kato \cite{Kato}.
\begin{Prop} \label{prop-Kato} 
For $\lambda \in P_+$ and $\mu \in Q_+$, we have
\[  \sum_{w \in W} (-1)^{\ell(w)} K^1_q(w \circ \lambda +\mu) = q^{-\langle \mu, \rho^\vee \rangle}P_{w_{\lambda - \mu}, w_{\lambda}}(q), \]
where $w_\nu$ is the element in the affine Weyl group $\hat W$ corresponding to $\nu \in P_+$ and $P_{w_{\lambda - \mu}, w_{\lambda}}(q)$ is the Kazhdan-Lusztig polynomial. 
\end{Prop}

Hence we obtain from Proposition \ref{prop-WW}:
\begin{Cor} \label{cor-KL}
\begin{eqnarray*} H_{\lambda+\rho} (\mu;q) &=&   H_{\lambda+\rho} (\mu;1)+ \dim V(\lambda)_{\lambda - \mu} - q^{-\langle \mu, \rho^\vee \rangle}P_{w_{\lambda - \mu}, w_{\lambda}}(q) \\ & & \phantom{LLLLLLLLLLLLLL} - \sum_{w \in  W \atop w \circ \lambda + \mu >0} (-1)^{\ell(w)}  \sum_{0 < \nu <w \circ \lambda +\mu} K^\infty_q(\nu ) K^1_q(w \circ \lambda +\mu -\nu).\end{eqnarray*} 
\end{Cor}

Setting $q=1$, and noting that $K_1^\infty(\beta)=0$ if $\beta>0$, we see the famous property of the Kazhdan-Lusztig polynomial:  $\dim V(\lambda)_{\lambda - \mu} = P_{w_{\lambda - \mu}, w_{\lambda}}(1)$.

\end{document}